\documentclass[11pt,epsf]{article}
\usepackage{amsthm}
\usepackage{amsfonts}
\usepackage{graphicx}
\usepackage{amssymb}
\title{On the double zeros of a partial theta function}
\author{Vladimir Petrov Kostov\\ 
Universit\'e de Nice, 
Laboratoire de Math\'ematiques, Parc Valrose,\\ 06108 Nice Cedex 2, France,  
e-mail: vladimir.kostov@unice.fr} 
\date{}
\newtheorem{tm}{Theorem}

\newtheorem{rem}[tm]{Remark}

\newtheorem{lm}[tm]{Lemma}

\newtheorem{prop}[tm]{Proposition}
\newtheorem{nota}[tm]{Notation}
\begin{document} 
\maketitle 

\begin{abstract}
The series $\theta (q,x):=\sum _{j=0}^{\infty}q^{j(j+1)/2}x^j$ 
converges for $q\in [0,1)$, $x\in \mathbb{R}$, 
and defines a {\em partial theta function}. For any fixed 
$q\in (0,1)$ it has infinitely many negative zeros. 
For $q$ taking one of the {\em spectral} values $\tilde{q}_1$, 
$\tilde{q}_2$, $\ldots$   
(where $0.3092493386\ldots =\tilde{q}_1<\tilde{q}_2<\cdots <1$, 
$\lim _{j\rightarrow \infty}\tilde{q}_j=1$) the function $\theta (q,.)$ 
has a double zero $y_j$ which is the rightmost of its real zeros 
(the rest of them being simple). 
For $q\neq \tilde{q}_j$ the partial theta function has no multiple real zeros.  
We prove that  
$\tilde{q}_j=1-\pi /2j+(\log j)/8j^2+O(1/j^2)$  
and $y_j=-e^{\pi}e^{-(\log j)/4j+O(1/j)}$.\\   
%where $b\in [-2.707051891\ldots ,-1.124681043\ldots ]$ and   
%$\alpha =-\pi /4+2b+\pi ^2/4$ hence  
%$\alpha \in$ $[-3.732100844\ldots ,-0.5673591490\ldots ]$.\\ 


\end{abstract}

\section{Introduction}

Consider the bivariate series $\theta (q,x):=\sum _{j=0}^{\infty}q^{j(j+1)/2}x^j$. 
For each fixed $q$ of the open unit disk it defines an entire function in $x$ 
called a {\em partial theta function}. This terminology is 
explained by the fact that the Jacobi theta function is the sum of the series 
$\sum _{j=-\infty}^{\infty}q^{j^2}x^j$ and one has 
$\theta (q^2,x/q)=\sum _{j=0}^{\infty}q^{j^2}x^j$. There are different domains 
in which the function $\theta$ finds applications: 
asymptotic analysis (\cite{BeKi}), statistical physics 
and combinatorics (\cite{So}), Ramanujan type $q$-series 
(\cite{Wa}) and the theory 
of (mock) modular forms (\cite{BrFoRh}). See also \cite{AnBe} 
for more information about this function. 

The function $\theta$ satisfies the following functional equation:

\begin{equation}\label{FE}
\theta (q,x)=1+qx\theta (q,qx)
\end{equation}

In what follows we consider $q$ as a parameter and $x$ as a variable. 
We treat only the case $q\in (0,1)$, $x\in \mathbb{R}$. For each fixed $q$ 
the function $\theta (q,.)$ has infinitely many negative zeros. (It has no 
positive zeros because its Taylor coefficients are all positive.) 
There exists a 
sequence of values $\tilde{q}_j$ of $q$ (called {\em spectral values}) such 
that $0.3092493386\ldots=\tilde{q}_1<\tilde{q}_2<\cdots$ (where  
$\tilde{q}_j\rightarrow 1^-$ as $j\rightarrow \infty$) for which 
and only for which the function $\theta (q,.)$ has a multiple real zero $y_j$ 
(see \cite{KoSh} and \cite{Ko2}). 
This zero is negative, of multiplicity $2$ and is 
the rightmost of its real zeros. The rest of them are simple. The function 
$\theta (\tilde{q}_j,.)$ has a local minimum at $y_j$. As $q$ increases and 
passes from $\tilde{q}_j^-$ to $\tilde{q}_j^+$, the rightmost two real zeros 
coalesce and give birth to a complex conjugate pair.

The double zero of $\theta (\tilde{q}_1,.)$ equals $-7.5032559833\ldots$. 
The spectral value $\tilde{q}_1$ is of interest in the context of 
a problem due to Hardy, Petrovitch and Hutchinson, 
see \cite{Ha}, \cite{Pe}, \cite{Hu}, \cite{Ost}, 
\cite{KaLoVi} and~\cite{KoSh}. The following asymptotic formula and limit  
are proved in ~\cite{Ko3}:

\begin{equation}\label{asold}
\tilde{q}_j=1-\pi /2j+o(1/j)~~,~~
\lim _{j\rightarrow \infty}y_j=-e^{\pi}=-23.1407\ldots
\end{equation}
In the present paper we make this result more precise:

\begin{tm}\label{basictm}
%(1) The spectral numbers $\tilde{q}_j$ can be expanded in convergent 
%power series of $1/j$ and $(\log j)/j$.
The following asymptotic estimates hold true:

\begin{equation}\label{asnew}
\begin{array}{rcl}\tilde{q}_j&=&1-\pi /2j+(\log j)/8j^2+b/j^2+o(1/j^2)\\ \\ 
y_j&=&-e^{\pi}e^{-(\log j)/4j+\alpha /j+o(1/j)}~~,\end{array}
\end{equation} 
$$\begin{array}{llcl}
{\it where}~~~~&b\in [1.735469700\ldots ,3.327099360\ldots ]&
~~~~{\it and}~~~~&  
\alpha =-\pi /4-2b+\pi ^2/4\\ \\ 
{\it hence}~~~~& 
\alpha \in [-4.972195782\ldots ,-1.788936462\ldots ]~.&&\end{array}$$ 
\end{tm} 

The first several numbers $y_j$ form a monotone decreasing sequence. 
We list the first five of them:

$$-7.5\ldots~~,~~-11.7\ldots~~,~~-14.0\ldots~~,~~-15.5\ldots~~,
~~-16.6\ldots~~.$$
The theorem implies that for $j$ large enough 
the sequence must also be decreasing and gives an idea about the rate 
with which the sequences $\{ \tilde{q}_j\}$ and $\{ y_j\}$ tend to their 
limit values. 
\vspace{2mm}

{\bf Acknowledgement.} The author has discussed partial theta functions with 
B.Z.~Shapiro during his visits to the University of Stockholm and with 
A.~Sokal by e-mail. V.~Katsnelson has sent to the author an earlier version 
of his paper~\cite{Ka}. To all of them 
the author expresses his most sincere gratitude.

\section{Plan of the proof of 
Theorem~\protect\ref{basictm}}

For $q\in (0,\tilde{q}_1)$ the function $\theta (q,.)$ has only simple 
negative zeros which we denote by $\xi _j$, where $\cdots <\xi _2<\xi _1<0$ 
(see \cite{KoSh} and \cite{Ko2}). The following 
equality holds true for all values of $q\in (0,1)$ (see~\cite{Ko4}):

\begin{equation}\label{product}
\theta (q,x)=\prod _{j=1}^{\infty}(1-x/\xi _j)
\end{equation} 
For $q\in [\tilde{q}_1,1)$ the indexation of the zeros is 
such that zeros change continuously as $q$ varies. 

For $q\in (0,\tilde{q}_1)$ all derivatives 
$(\partial ^k\theta /\partial x^k)(q,.)$ have only simple negative zeros. 
For $k=1$ this means that the numbers $t_s$ and $w_s$, where the function 
$\theta (q,.)$ has respectively local minima and maxima, satisfy the string 
of inequalities

\begin{equation}\label{string1}
\cdots <t_{s+1}<\xi _{2s+1}<w_s<\xi _{2s}<t_s<\xi _{2s-1}<\cdots <0~.
\end{equation}
The above inequalities hold true for any $q\in (0,1)$ 
whenever $\xi _{2s-1}$ is real negative (which implies that this is also 
the case of $\xi _j$ for $j>2s-1$). 

\begin{lm}\label{lmtsws}
Suppose that $q\in (\tilde{q}_j,\tilde{q}_{j+1}]$ (we set 
$\tilde{q}_0=0$). Then for $s\geq j+1$ one has $t_{s+1}\leq w_s/q$ and 
$w_s\leq t_s/q$.
\end{lm}

\begin{proof}
Equality (\ref{FE}) implies 

\begin{equation}\label{qxqxq}
(\partial \theta /\partial x)(q,x)=xq^2(\partial \theta /\partial x)(q,qx)
+\theta (q,qx)~.
\end{equation}
When $qx=t_s$ and $s\geq j+1$, then $\theta (q,t_s)\leq 0$,  
$(\partial \theta /\partial x)(q,t_s)=0$ and 
$(\partial \theta /\partial x)(q,t_s/q)\leq 0$. Hence the local maximum is 
to the left of or exactly at $t_s/q$, i. e. $w_s\leq t_s/q$. 

In the same way if one sets $qx=w_s$, one gets $\theta (q,w_s)\geq 0$,  
$(\partial \theta /\partial x)(q,w_s)=0$ and 
$(\partial \theta /\partial x)(q,w_s/q)\geq 0$ and hence $t_{s+1}\leq w_s/q$.
\end{proof} 

\begin{nota}\label{notar}
{\rm (1) In what follows we are using the numbers 
$u_s:=-q^{-2s+1/2}$ and $v_s:=-q^{-2s-1/2}$.

(2) We denote by $\tilde{r}_s$ the solution to the equation 
$\theta (q,u_s)=0$ and we set $z_s:=-(\tilde{r}_s)^{-2s+1/2}$.}
\end{nota} 

\begin{rem}
{\rm For $s$ sufficiently large the equation $\theta (q,u_s)=0$ has 
a unique solution. 
This follows from part (2) of Lemma~\ref{lmpsi}.}
\end{rem}

It is shown in \cite{Ko2} that

\begin{equation}\label{string2}
\cdots<\xi _4<u_2<\xi _3<-q^{-3}<\xi _2/q<v_1<\xi _1/q<-q^{-2}<\xi _2
<u_1<\xi _1<-q^{-1}<0~.
\end{equation}
These inequalities hold true for $q>0$ sufficiently small and for 
any $q\in (0,1)$ if the index $j$ of $\xi _j$ is sufficiently large.  

Comparing the inequalities (\ref{string1}) 
and (\ref{string2}) we see that $\xi _{2s+1}<w_s,v_s<\xi _{2s}$ and 
$\xi _{2s}<t_s,u_s<\xi _{2s-1}$. In this sense we say that a number $u_s$ 
(resp. $v_s$) corresponds to a local minimum (resp. maximum) 
of $\theta (q,.)$.

We prove the following theorems respectively in 
Sections~\ref{prrztm} and~\ref{prrqrtm}:

\begin{tm}\label{rztm}
The following asymptotic estimates hold true:
 
\begin{equation}\label{asnewrz}
\begin{array}{rcl}\tilde{r}_j&=&1-\pi /2j+(\log j)/8j^2+b^*/j^2+o(1/j^2)\\ \\ 
z_j&=&-e^{\pi}e^{-(\log j)/4j+\alpha ^*/j+o(1/j)}~~,\end{array}
\end{equation}
$$\begin{array}{llcl}
{\it where}~~~~&b^*\in [1.735469700\ldots ,1.756303033\ldots ]&
~~~~{\it and}~~~~& 
\alpha ^*=-\pi /4-2b^*+\pi ^2/4\\ \\ 
{\it hence}~~~~&\alpha ^*\in 
[-1.830603128\ldots ,-1.788936462\ldots ]~.&&\end{array}$$
\end{tm}

\begin{tm}\label{rqrtm}
For $j$ sufficiently large one has 
$0<\tilde{r}_j\leq \tilde{q}_j\leq \tilde{r}_{j+1}<1$.
\end{tm}

Theorem~\ref{basictm} follows from the above two theorems. Indeed, as 

$$\begin{array}{rccccc}
1-\pi /2j+(\log j)/8j^2+O(1/j^2))&=&
\tilde{r}_j&\leq &\tilde{q}_j&{\rm and}\\ \\ 
1-\pi /2(j+1)+(\log (j+1))/8(j+1)^2+O(1/(j+1)^2)&=&\tilde{r}_{j+1}&\geq& 
\tilde{q}_j\end{array}$$
one deduces immediately the equality $\tilde{q}_j=1-\pi /2j+o(1/j)~(A)$. It is 
also clear that $\tilde{r}_{j+1}=1-\pi /2j+(\log j)/8j^2+O(1/j^2)$.  
%from which the first of the formulas (\ref{asnew}) follows. 

To obtain an estimate of the term $O(1/j^2)$ recall that 

$$\begin{array}{cclc}
\tilde{r}_j&=&1-\pi /2j+(\log j)/8j^2+b^*/j^2+o(1/j^2)&{\rm hence}\\ \\ 
\tilde{r}_{j+1}&=&1-\pi /2j+(\log j)/8j^2+(b^*+\pi /2)/j^2+o(1/j^2)
\end{array}$$ 
(we use the equality $1/(j+1)=1/j-1/j(j+1)$). 
This implies $\tilde{q}_j=1-\pi /2j+(\log j)/8j^2+b/j^2+o(1/j^2)$, 
where $b\in [b^*,b^*+\pi /2]$ hence 
$b\in [1.735469700\ldots ,3.327099360\ldots ]$. The 
quantities $\alpha$ and $\alpha ^*$ 
are expressed by similar formulas via $b$ and $b^*$, 
see Theorems~\ref{basictm} and \ref{rztm}. 
This gives the closed intervals to which 
$\alpha$ and $\alpha ^*$ belong.

Section \ref{secpsi} contains properties of the function $\psi$ used in the 
proofs. At first reading one can read only the statements of 
Theorem~\ref{recallpsi} and Proposition~\ref{moreprecise} from that section.

\section{Proof of Theorem~
\protect\ref{rqrtm}\protect\label{prrqrtm}}

We prove first the inequality $\tilde{r}_j\leq \tilde{q}_j$. When $q$ increases 
and becomes equal to a spectral value $\tilde{q}_j$, then two 
negative zeros of $\theta$ coalesce. The corresponding double zero 
of $\theta (\tilde{q}_j,.)$ is a local minimum. It equals $t_j$. Hence for 
some value of $q$ not greater than $\tilde{q}_j$ one has 
$\theta (q,u_j)=0$. This value is $\tilde{r}_j$, see Notation~\ref{notar}. 

To prove the inequality $\tilde{q}_j\leq \tilde{r}_{j+1}~(*)$ 
we use a result due to V. Katsnelson, see~\cite{Ka}: 
\vspace{2mm}

{\em The sum of the series 
$\sum _{j=0}^{\infty}q^{j(j+1)/2}x^j$ (considered for $q\in (0,1)$ and $x$ 
complex) tends to $1/(1-x)$ 
(for $x$ fixed and as $q\rightarrow 1^-$)  
exactly when 
$x$ belongs to the interior of the closed Jordan curve 
$\{ e^{|s|+is}, s\in [-\pi ,\pi]\}$.} 
\vspace{2mm}

Hence in particular $\theta (q,x)$ converges to $1/(1-x)$ as 
$q\rightarrow 1^-$ for each fixed $x\in (-e^{\pi},0]$. This means that for 
$j$ sufficiently large one has $y_j<-23$ (because the function $1/(1-x)$ has 
no zeros on $(-\infty ,0]$ and $23<e^{\pi}$).   
 
\begin{prop}\label{propqv}
For $j$ sufficiently large one has $\theta (\tilde{q}_j,v_j)>1/3$.
\end{prop}

Before proving the proposition we deduce the inequality $(*)$ from it. Recall 
that $\tilde{q}_j\geq \tilde{q}_1>0.3$. Using 
equation (\ref{FE}) one gets 

$$\theta (\tilde{q}_j,u_{j+1})=1+\tilde{q}_ju_{j+1}\theta (\tilde{q}_j,v_j)
<1-0.3\times 23 \times (1/3)=-1.3<0~.$$
Hence for $q=\tilde{q}_j$ the value of $\theta (q,u_{j+1})$ is still 
negative, i. e. one has $\theta (q,u_{j+1})=0$ for some value $q>\tilde{q}_j$. 
   
\begin{proof}[Proof of Proposition~\ref{propqv}]

We deduce the proposition from the following two lemmas:

\begin{lm}\label{lm1/2}
Suppose that the quantity $v_j$ is computed for $q$ equal to the spectral value 
$\tilde{q}_j$. Set $\Xi _j:=(-\tilde{q}_j^{-2j}-v_j(\tilde{q}_j))/
(-\tilde{q}^{-2j}+\tilde{q}^{-2j-1})$. Then $\lim _{j\rightarrow \infty}\Xi _j=1/2$.
%One has $\lim _{j\rightarrow \infty}(-\tilde{q}_j^{-2j}-v_j(\tilde{q}_j))/
%(-\tilde{q}^{-2j}+\tilde{q}^{-2j-1})\rightarrow 1/2$. 
\end{lm}

The next lemma considers certain points 
of the graph of $\theta (\tilde{q}_j,.)$. Recall that 
$\theta (\tilde{q}_j,w_j)=1$ because for $q=\tilde{q}_j$ one has 
$\theta (\tilde{q}_j,t_j)=0$, $w_j=t_j/\tilde{q}_j$ and by equation (\ref{FE}) 
$\theta (\tilde{q}_j,w_j)=1+\tilde{q}_jw_j\theta (\tilde{q}_j,t_j)=1$.

\begin{lm}\label{above}
The point~~~$(v_j,\theta (\tilde{q}_j,v_j))$~~~lies above or on the straight 
line passing 
through the two points 
$(-\tilde{q}_j^{-2j},\theta (\tilde{q}_j,-\tilde{q}_j^{-2j}))$ 
and $(w_j,\theta (\tilde{q}_j,w_j))=(w_j,1)$.
\end{lm}

The two lemmas imply that for $j$ sufficiently large the following 
inequality holds true:

$$\theta (\tilde{q}_j,v_j)-\theta (\tilde{q}_j,-\tilde{q}_j^{-2j})\geq \Xi _j
(\theta (\tilde{q}_j,w_j)-\theta (\tilde{q}_j,-\tilde{q}_j^{-2j}))>
(1/3)(1-\theta (\tilde{q}_j,-\tilde{q}_j^{-2j}))~.$$
Hence $\theta (\tilde{q}_j,v_j)>1/3+
(2/3)\theta (\tilde{q}_j,-\tilde{q}_j^{-2j})$. 
It is shown in \cite{Ko2} (see Proposition~9 there) 
that for $q\in (0,1)$ one has 
$\theta (q,-q^{-s})\in (0,q^s)$, $s\in \mathbb{N}$. Hence 
$\theta (\tilde{q}_j,v_j)>1/3$.
\end{proof}   

\begin{proof}[Proof of Lemma \ref{lm1/2}]
It is clear that $(-\tilde{q}_j^{-2j}-v_j(\tilde{q}_j))/
(-\tilde{q}_j^{-2j}+\tilde{q}_j^{-2j-1})=
\sqrt{\tilde{q}_j}(1-\sqrt{\tilde{q}_j})/(1-\tilde{q}_j)=
\sqrt{\tilde{q}_j}/(1+\sqrt{\tilde{q}_j})$. As 
$j\rightarrow \infty$ one has $\tilde{q}_j\rightarrow 1$ and the above fraction 
tends to $1/2$.
\end{proof}

\begin{proof}[Proof of Lemma~\ref{above}]
We are going to prove a more general statement from which the lemma 
follows.
Suppose that $w_s\leq x_2<x_1<x_0\leq -q^{-2s}<-1<0$ for some 
$s\in \mathbb{N}$. 
Set $\theta (q,x_0)=C$, $\theta (q,x_1)=B$, $\theta (q,x_2)=A$. Suppose 
that $A>B>C>0$ and $A\geq 1$. We use the 
letters $A$, $B$ and $C$ also to denote the points of the graph of 
$\theta (q,.)$ 
with coordinates $(x_0, C)$, $(x_1, B)$ and $(x_2, A)$. 
%In what follows we 
%assume that $x_0=-q^{-2s+1}$, $x_2<x_1=-q^{-2s+1/2}$ and that $A=1$. 

{\em We prove that the point $B$ is above or on the straight line $AC$.} 

Indeed, suppose that the point $B$ is below the straight line $AC$. Then 

\begin{equation}\label{cond1}
\frac{B-C}{|x_1-x_0|}<\frac{A-B}{|x_2-x_1|}
\end{equation}
Consider the points $(x_0/q, C')$, $(x_1/q, B')$ and $(x_2/q, A')$ 
of the graph of $\theta (q,.)$. Equation (\ref{FE}) implies that 

$$C'=1+x_0C~~,~~B'=1+x_1B~~,~~A'=1+x_2A~.$$
In the same way for the points $(x_0/q^2, C'')$, $(x_1/q^2, B'')$ 
and $(x_2/q^2, A'')$ of the graph of $\theta (q,.)$ one gets 

\begin{equation}\label{GGG}
\begin{array}{cclcclccl}C''&=&1+\frac{x_0}{q}C'&~~~~B''&
=&1+\frac{x_1}{q}B&~~~~A''
&=&1+\frac{x_2}{q}A'\\ \\
&=&1+\frac{x_0}{q}+\frac{x_0^2}{q}C&&=&1+\frac{x_1}{q}+\frac{x_1^2}{q}B&&
=&1+\frac{x_2}{q}+\frac{x_2^2}{q}A~.\end{array}
\end{equation}

It is clear that $x_2/q^2<x_1/q^2<x_0/q^2\leq -q^{-2s-2}$. By Lemma~\ref{lmtsws}  
one has 

$$w_{s+1}\leq t_{s+1}/q\leq w_s/q^2\leq x_2/q^2\leq -q^{-2s-2}~.$$ 
The inequalities 
$A>B>C>0$ and $x_2<x_1<x_0<-1$ imply $A'<B'<C'$, see (\ref{GGG}). As 
$A\geq 1$ and 
$x_2<-1$, one has $A'<0$. Therefore $A''=1+(x_2/q)A'>1$. 

It follows from $w_{s+1}\leq x_2/q^2<x_1/q^2<x_0/q^2\leq -q^{-2s-2}$ that 
$B''>0$ and $C''>0$. (Indeed, $\theta (q,x)>0$ for $x\in (-q^{-2s-1},-q^{-2s})$, 
see \cite{Ko2}.)  
If $B'<C'<0$, then $x_2<x_1<x_0<-1$ implies 
$A''>B''>C''$, see (\ref{GGG}). If $B'<0\leq C'$, then again by (\ref{GGG}) 
one gets $A''>B''>C''$. 

If $0\leq B'<C'$, then one obtains $A''>B''$ and $A''>C''$. If $B''\leq C''$, 
then the point $B''$ lies below the straight line $A''C''$. In this case 
one can find a point $(x_1^{1}/q^2,{B_*}'')$ of the graph of $\theta (q,.)$ such 
that 

$$w_{s+1}\leq x_2/q^2<x_1^{1}/q^2<x_1/q^2<x_0/q^2\leq -q^{-2s-2}<-1<0~~~,~~~ 
A''>{B_*}''>C''>0$$ 
and the point ${B_*}''$ lies below the straight line $A''B''$. 

Suppose that $A''>B''>C''$. 
We show that the point $B''$ is below the straight line $A''C''$. This is 
equivalent to proving that 

$$\frac{(x_1/q)+(x_1^2/q)B-(x_0/q)-(x_0^2/q)C}{|x_1-x_0|/q^2}<
\frac{(x_2/q)+(x_2^2/q)A-(x_1/q)-(x_1^2/q)B}{|x_2-x_1|/q^2}$$
or to proving the inequality

\begin{equation}\label{cond2}
B(x_1^2/x_0^2)|x_2-x_0|-C|x_2-x_1|-A(x_2^2/x_0^2)|x_1-x_0|<0~.
\end{equation}
Inequality (\ref{cond1}) can be given another presentation:

\begin{equation}\label{cond3}
B|x_2-x_0|-C|x_2-x_1|-A|x_1-x_0|<0~.
\end{equation}
One can notice that inequality (\ref{cond2}) (which we want to prove) is 
the sum of inequality (\ref{cond3}) (which is true) and the inequality

\begin{equation}\label{cond4}
B\left( \frac{x_1^2}{x_0^2}-1\right) |x_2-x_0|-
A\left( \frac{x_2^2}{x_0^2}-1\right) |x_1-x_0|<0~.
\end{equation}
So if we show that inequality (\ref{cond4}) is true, then this will 
imply that inequality (\ref{cond2}) is also true. Recall that 
$x_2<x_1<x_0<0$ and $A>B>C>0$. Hence inequality (\ref{cond4}) is equivalent to 

$$B(|x_1|+|x_0|)-A(|x_2|+|x_0|)<0$$
which is obviously true. We set ${B_*}''=B''$ and $x_1^{1}=x_1$. 

For $s\in \mathbb{N}$ we define in the same way the three points 
$(x_0/q^{2s},C^{(2s)})$, $(x_1^{s}/q^{2s},{B_*}^{(2s)})$ and 
$(x_2/q^{2s},A^{(2s)})$ by the condition that they belong to the graph of 
$\theta (q,.)$, $w_s\leq x_2/q^{2s}<x_1^{s}/q^{2s}<x_0/q^{2s}<-q^{-2s}$, 
$A^{(2s)}>{B_*}^{(2s)}>C^{(2s)}>0$ and 
the point ${B_*}^{(2s)}$ lies below the 
straight line $A^{(2s)}C^{(2s)}$. 
%(In our proof we did not use the concrete values 
%of $x_0$, $x_1$ and $A$.)\epsfxsize=10cm 
\begin{figure}[htbp]
   \centerline{\hbox{\includegraphics[scale=0.8]{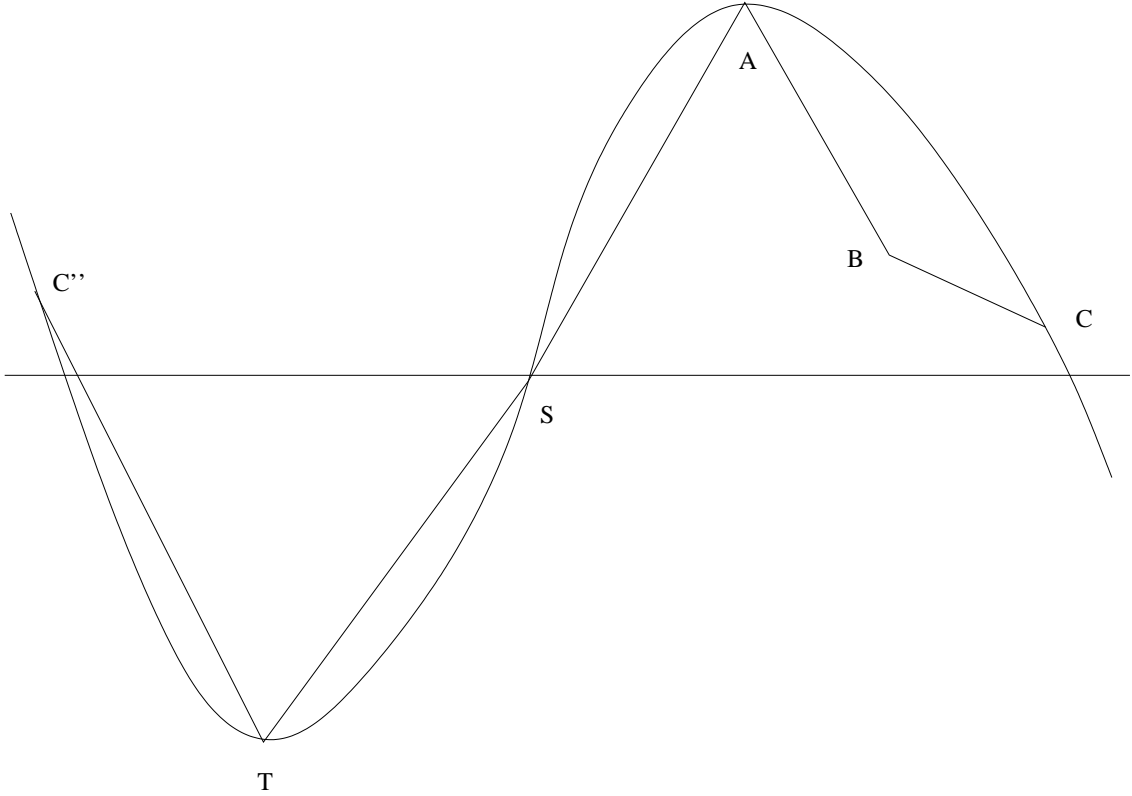}}}
    \caption{Part of the graph of a partial theta function and the points 
$A$, $B$, $C$ and $C''$.}
\label{Fig1}
\end{figure}

This implies that the graph of $\theta (q,.)$ has on each interval 
$(-q^{-2s-2N},-q^{-2s})$, $N\in \mathbb{N}$, at least $O(4N)$ inflection points,  
twice as much as $O(2N)$, the one that should be. (This contradiction 
proves the lemma.) Indeed, on Fig.~\ref{Fig1} 
we show part of the graph of $\theta (q,.)$ as it should look like 
(the sinusoidal curve)
and the points $C$, $B$, $A$ 
and $C''$. If the point $B$ 
is below the straight line $AC$, then the change of convexity requires 
two more inflection points between a local minimum of $\theta (q,.)$ and 
the local maximum to its left.     
\end{proof}

\section{The function $\psi$\protect\label{secpsi}}

In the present section we consider the function 
$\psi (q):=1+2\sum _{j=1}^{\infty}(-1)^jq^{j^2}$. It 
is real-analytic on $(-1,1)$. This function has been studied in \cite{Ko1} 
and the following theorem recalls the basic results about it. Part (1) is a 
well-known property while parts (2) -- (7) are proved in~\cite{Ko1}.

\begin{tm}\label{recallpsi}
(1) By the Jacobi triple product identity the function $\psi$ can 
be expressed as follows (see \cite{PoSz}, Chapter 1, Problem 56):

\begin{equation}\label{Jtriple}
\psi (q)=\prod _{j=1}^{\infty}\frac{1-q^j}{1+q^j}
\end{equation}

(2) The function $\psi$ is decreasing, i.e. $\psi '<0$ for all $q\in (-1,1)$.

(3) For the endpoints of its interval of definition one has the limits 
$\lim _{q\rightarrow 1^-}\psi (q)=0$, 
$\lim _{q\rightarrow -1^+}\psi (q)=+\infty$.

(4) The function $\psi$ is flat at $1$, i.e. for any $l\in \mathbb{N}$, 
$\psi (q)=o((q-1)^l)$ as $q\rightarrow 1^-$. 

(5) The function $\psi$ is convex, i.e. 
$\psi ''\geq 0$ for all $q\in (-1,1)$, with equality only for $q=0$.

(6) Consider the function $\tau (q):=(q-1)\log \psi (q)$. It 
is increasing on $(0,1)$ 
and $\lim _{q\rightarrow 1^-}\tau (q)=\pi ^2/4$. 
This implies that for any $\varepsilon >0$ there exists $\delta >0$ such that 
$\displaystyle{e^{\frac{\pi ^2}{4(q-1)}}<\psi (q)\leq e^{\frac{\pi ^2-\varepsilon}{4(q-1)}}}$ 
for $q\in (1-\delta ,1)$.

(7) As $q\rightarrow -1^+$, the growth rate of the function $\psi$ 
satisfies the conditions $\psi (q)=o((q+1)^{-1})$ and 
$(q+1)^{\alpha}/\psi (q)=o(1)$ for any $\alpha \in (-1,0)$. 
\end{tm}

Set $D:=1/2+\log 2 +\pi ^2/8=2.426847731\ldots$. 
Property (6) can be further detailed: 

\begin{prop}\label{moreprecise}
For $q$ close to $1$ the function $\tau$ is of the form 

$$\tau =\pi ^2/4+(1/2)(1-q)\log (1-q)-K(1-q)+o(1-q)~~~~{\rm with}$$

$$K\in [D,D+1/12]=[2.426847731\ldots ,2.510181064\ldots ]~.$$
%(1/2+\log 2-\pi ^2/8)(1-q)+O((1-q)^2\log (1-q))~.$$ 
Hence $\displaystyle{\psi =e^{\frac{\tau}{q-1}}=e^K(1+o(1))\, (1-q)^{-\frac{1}{2}}\, 
e^{\frac{\pi ^2}{4(q-1)}}}$.
\end{prop}

\begin{proof}
The logarithm of the $j$th factor of 
the right-hand side of formula (\ref{Jtriple}) equals 

\begin{equation}\label{Taylorlog}
\log \frac{1-q^j}{1+q^j}=(-2)\left( q^j+\frac{q^{3j}}{3}+
\frac{q^{5j}}{5}+\cdots \right) ~.
\end{equation} 
This means that 

\begin{equation}\label{Taylorlog1}
\begin{array}{ccl}\log \psi (q)&:=&
\displaystyle{(-2)\sum _{j=1}^{\infty}\left( q^j+\frac{q^{3j}}{3}+
\frac{q^{5j}}{5}+\cdots \right) }\\ \\ 
&=&\displaystyle{(-2)\sum _{k=0}^{\infty}\frac{q^{2k+1}}{(2k+1)(1-q^{2k+1})}}~.
\end{array}
\end{equation}
Hence $\tau (q):=(q-1)\log \psi (q)=2\sum _{k=0}^{\infty}\zeta _k(q)$, where

$$\displaystyle{\zeta _k(q):=\frac{q^{2k+1}}{(2k+1)(1+q+q^2+\cdots +
q^{2k})}=\frac{q^{2k+1}(1-q)}{(2k+1)(1-q^{2k+1})}}~.$$

\begin{lm}\label{lmtau1}
For $q\in (0,1]$ the following inequalities hold true: 

$$q^{2k+1}/(2k+1)^2\leq \zeta _k(q)\leq q^{k+1}/(2k+1)^2$$ 
with equalities only for $q=1$.
\end{lm}

\begin{proof}[Proof of Lemma~\protect\ref{lmtau1}]
The inequalities result from $1+q+\cdots +q^{2k}\leq 2k+1$ and 
$q^{k+j}+q^{k-j}\geq 2q^k$ hence $1+q+\cdots +q^{2k}\geq (2k+1)q^k$ (with 
equalities only for $q=1$). 
\end{proof}

The above lemma gives the idea to compare the function $\tau$ 
(for $q$ close to $1$) with the function 
$h(q):=2\sum _{k=0}^{\infty}q^{k+1}/(2k+1)^2$. The lemma implies the following 
result:

\begin{equation}\label{twoinequalities}
h(q^2)/q\leq \tau (q)\leq h(q)~~,~~(\tau (q)=h(q))~\Longleftrightarrow ~(q=1)
\end{equation}
Our next step is to compare the asymptotic expansions of the functions $\tau$ 
and $h$ close to~$1$:

\begin{lm}\label{lmtau2}
For $q$ close to $1$ the following equality holds:
\begin{equation}\label{equh}
h(q)=\pi ^2/4+(1/2)(1-q)\log (1-q)-
D(1-q)+O((1-q)^2\log (1-q))
\end{equation}
%\begin{equation}\label{equtauh}
%|\tau (q)-h(q)|=O(1-q)
%\end{equation}
\end{lm}

%Proposition~\ref{moreprecise} results from the lemma. 
%\end{proof}

\begin{proof}[Proof of Lemma~\protect\ref{lmtau2}]
Notice first that 
$\lim _{q\rightarrow 1^-}\tau (q)=h(1)=\pi ^2/4$ and that 
$h=h_1+h_2$, where 

$$h_1:=2\sum _{k=0}^{\infty}q^{k+1}/(2k+1)(2k+2)~~~ 
{\rm and}~~~h_2:=2\sum _{k=0}^{\infty}q^{k+1}/(2k+1)^2(2k+2)~.$$ 
Equation (\ref{Taylorlog}) implies
$\log ((1+q)/(1-q))=
2\sum _{k=0}^{\infty}q^{2k+1}/(2k+1)$. Integrating both sides of this equality 
yields 

$$(1+q)\log (1+q)+(1-q)\log (1-q)=
2\sum _{k=0}^{\infty}q^{2k+2}/(2k+1)(2k+2)~.$$ 
Thus
$h_1=(1+q^{1/2})\log (1+q^{1/2})+(1-q^{1/2})\log (1-q^{1/2})$. The first summand 
is real analytic in a neighbourhood of $1$ and equals 
$2\log 2-(1/2)(1+\log 2)(1-q)+O((1-q)^2)$. The second one is equal to

$$\begin{array}{lc}
[\, (1-q)/(1+q^{1/2})\, ]\, (\log (1-q)-\log (1+q^{1/2}))&=\\ \\ 
(1-q)\, [\, (1/2+O(1-q))\log (1-q)-(\log (1+q^{1/2}))/(1+q^{1/2})\, ]&=\\ \\ 
(1/2)(1-q)\log (1-q)-(1/2)(\log 2)(1-q)+O((1-q)^2\log (1-q))~.\end{array}$$
About the function $h_2$ one can notice that there exist the limits 
$\lim _{q\rightarrow 1^-}h_2$ and 
$\lim _{q\rightarrow 1^-}h_2'$ (the latter equals $\pi ^2/8$). 
This implies formula~(\ref{equh}).
\end{proof}

\begin{lm}\label{lmtau3}
For $q\in (0,1)$ it is true that 
\begin{equation}\label{equlmtau3}
h(q)-\tau (q)\leq (1-q)/12~.
\end{equation}
\end{lm}

Proposition~\ref{moreprecise} results from the last two lemmas.
\end{proof}
\begin{proof}[Proof of Lemma~\protect\ref{lmtau3}]
To prove formula~(\ref{equlmtau3})  
set $R:=1/(2k+1)^2(1+q+\cdots +q^{2k})$ and $S_l:=1+q+\cdots +q^l$. Hence 

\begin{equation}\label{several}
\begin{array}{rcl}
\displaystyle{\frac{1}{2k+1}\left(\frac{q^{k+1}}{2k+1}- 
\frac{q^{2k+1}}{1+q+\cdots +q^{2k}}\right)}&=&
(q^{k+1}(1+q+\cdots +q^{2k})-(2k+1)q^{2k+1})R\\
&=&\displaystyle{\left( \sum _{j=0}^{k-1}(q^{k+1+j}+q^{3k+1-j}-2q^{2k+1})\right) R}\\
&=&\displaystyle{\left( \sum _{j=0}^{k-1}q^{k+1+j}(1-q^{k-j})^2\right) R}\\ 
&=&\displaystyle{q^{k+1}(1-q)^2\left( \sum _{j=0}^{k-1}q^j
(1+q+\cdots +q^{k-j-1})^2\right) R}\\ 
&=&\displaystyle{q^{k+1}(1-q)^2\left( 
\sum _{j=0}^{k-1}q^j\left( 
\sum _{\nu =0}^{k-j-1}q^{\nu}S_{2k-2j-2-2\nu}\right) \right) R~.}
\end{array}
\end{equation}
The sums $S_l$ enjoy the following property:

\begin{equation}\label{S_l}
(l-1)S_l\geq (l+1)qS_{l-2}~.
\end{equation}
Indeed, this is equivalent to $(l-1)(1+q^l)\geq 2qS_{l-2}$. The last inequality 
follows from $1+q^r\geq q+q^{r-1}$ (i.e. $(1-q)(1-q^{r-1})\geq 0$) applied for 
suitable choices of the exponent $r$. Equation (\ref{S_l}) implies the 
next property (whenever the indices are meaningful):

\begin{equation}\label{S_lbis}
(l-2\nu +1)S_l\geq (l+1)q^{\nu}S_{l-2\nu}~.
\end{equation}
Using equation (\ref{S_lbis}) one can 
notice that the right-hand side of (\ref{several}) is not larger than 
  
$$\begin{array}{cl}&\displaystyle{q^{k+1}(1-q)^2\left( 
\sum _{j=0}^{k-1}q^j\left( 
\sum _{\nu =0}^{k-j-1}\frac{2k-2j-1-2\nu}{2k-2j-1}S_{2k-2j-2}\right) \right) R}\\
\leq &\displaystyle{q^{k+1}(1-q)^2\left( 
\sum _{j=0}^{k-1}\left( 
\sum _{\nu =0}^{k-j-1}\left( \frac{2k-2j-1-2\nu}{2k-2j-1}\right) \left(  
\frac{2k-2j-1}{2k-1}\right) S_{2k-2}\right) \right) R}\\
=&\displaystyle{q^{k+1}(1-q)^2\left( \sum _{j=0}^{k-1}\sum _{\nu =0}^{k-j-1}
\frac{2k-2j-1-2\nu}{2k-1}S_{2k-2}\right) R}\\
=&\displaystyle{q^{k+1}(1-q)^2
\left( \sum _{j=0}^{k-1}\frac{(k-j)^2}{2k-1}S_{2k-2} \right) R}\\ 
=&\displaystyle{q^{k+1}(1-q)^2\frac{k(k+1)(2k+1)}{6(2k-1)(2k+1)^2}
\frac{S_{2k-2}}{S_{2k}}}\\
\leq&\displaystyle{q^k(1-q)^2\frac{k(k+1)}{6(2k+1)^2}~.}
\end{array}$$
At the last line we used property (\ref{S_l}) with $l=2k$. 
The last fraction is less 
than $1/24$. Hence $h(q)-\tau (q)\leq (1-q)^2\sum _{k=0}^{\infty}q^k/12=(1-q)/12$.
\end{proof}

\section{Proof of Theorem~\protect\ref{rztm}
\protect\label{prrztm}}

We follow the same path of reasoning as the one used in \cite{Ko3}. In this section we use the results of \cite{Ko3} and \cite{Ko1}. 
Set 

$$\lambda _s(q):=\sum _{j=2s}^{\infty}(-1)^jq^{j^2/2}~,~ 
\chi _s(q):=\lambda _s(q)/q^{2s^2}=\sum _{j=2s}^{\infty}(-1)^jq^{j^2/2-2s^2}=
\sum _{j=0}^{\infty}(-1)^jq^{(j^2+4js)/2}~.$$ 
The equation $\theta (q,-q^{-2s+1/2})=0$ is equivalent to (see~\cite{Ko3}) 

\begin{equation}\label{psilambda1}
\psi (q^{1/2})=\lambda _s(q)=q^{2s^2}\chi _s(q)~.
\end{equation} 
The following lemma is also proved in \cite{Ko3}:

\begin{lm}\label{lmpsi}
(1) One has $\lim _{q\rightarrow 1^-}\lambda _s(q)=
\lim _{q\rightarrow 1^-}\chi _s(q)=1/2$. 

(2) For $s\in \mathbb{N}$ sufficiently large the graphs of the functions 
$\psi (q^{1/2})$ and $\lambda _{s}(q)$ 
(considered for $q\in [0,1]$) intersect at exactly one point belonging to 
$(0,1)$ and at $1$.

(3) For $q\in [0,1]$ 
the inequality $\lambda _s(q)\geq \lambda _{s+1}(q)$ holds true  
with equality for $q=0$ and $q=1$.

(4) For $q\in [0,1]$ one has $1/2\leq \chi _s(q)\leq 1$.  
\end{lm} 

Part (2) 
of the lemma implies that for each $s$ sufficiently large the number 
$\tilde{r}_s$ is correctly 
defined. Part (3) implies that the numbers $\tilde{r}_s$ form an increasing 
sequence. Indeed, this follows from $\psi (q^{1/2})$ being a decreasing 
function, see part (2) of Theorem~\ref{recallpsi}. 

Recall that the constant $K$ was introduced by Proposition~\ref{moreprecise}. 
Set $q:=\tilde{r}_s=1-h_s/s$. Consider the equalities (\ref{psilambda1}). The 
left-hand side is representable in the form 
%$e^{-((\pi ^2-\varepsilon _s)/4)(2s/h_s)}$ (where $\varepsilon _s\geq 0$ and 
%$%\lim _{s\rightarrow \infty}\varepsilon _s=0$, 
%see part (6) of Theorem~\ref{recallpsi}). 
$$\displaystyle{e^K(1+o(1))\, [\, (1-q)^{-1/2}/
(1+\sqrt{q})^{-1/2}\, ]\, e^{\pi ^2(\sqrt{q}+1)/4(q-1)}~,}$$
see Proposition~\ref{moreprecise}.
Hence $\log \psi (q^{1/2})$ is of the form

$$\begin{array}{cl}&
(\pi ^2/4)(-s/h_s)(2-(1/2)(h_s/s)+o(h_s/s))-(1/2)\log (h_s/s)+
(1/2)\log 2+K +o(1)\\ \\ =&
-(\pi ^2/2)(s/h_s)+(1/2)\log s-(1/2)\log h_s+L+o(1)~,\end{array}$$
where $L:=K+(1/2)\log 2+\pi ^2/8$. 
The right-hand side of (\ref{psilambda1}) equals 
$(1-h_s/s)^{2s^2}\chi _s(1-h_s/s)$. 

Hence its logarithm is of the form

$$\begin{array}{rcl}(2s^2)\log (1-h_s/s)+\log (\chi _s(1-h_s/s))&=&
-(2s^2)(h_s/s+h_s^2/2s^2+O(1/s^3))-\log 2+o(1)\\ \\ 
&=&-2sh_s-(h_s)^2-\log 2+o(1)~.\end{array}$$
(we use $\chi _s=1/2+o(1)$, see part (1) of Lemma~\ref{lmpsi}; hence 
$\log h_s=\log (\pi /2)+o(1)$). 
Set $h_s:=\pi /2+d_s$. Hence $d_s=o(1)$, see equality $(A)$ 
after Theorem~\ref{rqrtm}, and
%(\ref{asold}) and \cite{Ko3}, and  

$$\begin{array}{cl}&
-(\pi ^2/2)(s/(\pi /2+d_s))+(1/2)\log s-(1/2)\log (\pi /2)+L+o(1)\\ \\=&
-2s(\pi /2+d_s)-(\pi /2+d_s)^2-\log 2+o(1)\end{array}$$
or equivalently
\begin{equation}\label{longequ}
\begin{array}{cl}&-(\pi ^2/2)s+(\pi /2+d_s)((\log s)/2-(\log (\pi /2))/2+L)+o(1)\\ \\=&
-2s(\pi /2+d_s)^2-(\pi /2+d_s)^3-(\pi /2+d_s)\log 2+o(1)\end{array}
\end{equation}
The terms $-\pi ^2/2s$ cancel. Hence 

$$((\pi /2+d_s)/2)\log s=-2s\pi d_s-2s(d_s)^2+O(1)$$ 
i.e. $d_s=-((\log s)/8s)(1+o(1))$. Set $d_s:=-(\log s)/8s+g_s$. Using 
equation (\ref{longequ}) one gets 

\begin{equation}\label{TTTT}
\begin{array}{cl}
&-(\pi ^2/2)s+(\pi /2-(\log s)/8s+g_s)((\log s)/2-(\log (\pi /2))/2+L)+o(1)\\ \\=&-2s(\pi /2-(\log s)/8s+g_s)^2-(\pi /2-(\log s)/8s+g_s)^3-(\pi /2-(\log s)/8s+g_s)\log 2\end{array}
\end{equation}
To find the main asymptotic term in the expansion of $g_s$ we have to leave 
only the linear terms in $g_s$ and the terms 
independent of $g_s$ (because $g_s^2=o(g_s)$).   
%Of the latter it suffices to leave the constant term 
%because the terms $s$ and $\log s$ cancel. 
%We replace $h_s$ by $\pi /2$ 
%because $h_s=\pi /2+o(1)$. 
The left-hand side of equation (\ref{TTTT}) takes the form:

$$-(\pi ^2/2)s+(\pi /4)(\log s)-(\pi /4)\log (\pi /2)+(\pi /2)L+
g_s((\log s)/2)(1+o(1))+o(1)~.$$
The right-hand side equals

$$-2s(\pi /2)^2+\pi (\log s)/4-2s\pi g_s-\pi ^3/8-(3\pi ^2/4)g_s
-(\pi /2)\log 2-g_s\log 2+o(1)~.$$
The terms $s$ and $\log s$ cancel. The remaining terms give the equality

$$\begin{array}{ll}
&(((\log s)/2)(1+o(1))+2s\pi +O(1))g_s\\ \\ =&
(\pi /4)\log (\pi /2)-(\pi /2)L-\pi ^3/8-(\pi /2)\log 2+o(1)~.
\end{array}$$
Hence $g_s=(1/s)(M+o(1))$, where 

$$M:=(\log (\pi /2))/8-L/4-\pi ^2/16-(\log 2)/4=
(\log (\pi /8))/8-L/4-\pi ^2/16~.$$ 
Now recall that 

$$L=K+(1/2)\log 2+\pi ^2/8~~~~{\rm and}~~~~ 
K\in [2.426847731\ldots ,2.510181064\ldots ]~.$$ 
Hence 
$M=-b^*=(\log (\pi /16))/8-3\pi ^2/32-K/4$, 
so $b^*\in [1.735469700\ldots ,1.756303033\ldots ]$.

Thus we 
have proved the first of formulas (\ref{asnewrz}). To prove the second one 
it suffices to notice that 

$$z_s=-(\tilde{r}_s)^{-2s+1/2}=
-(1-\pi /2s+(\log s)/8s^2+b^*/s^2+\cdots )^{-2s+1/2}~.$$ 
Set $\Phi :=\pi /2s-(\log s)/8s^2-b^*/s^2+\cdots$. 
Hence 
$$z_s=-e^{(-2s+1/2)\log (1-\Phi )}=-e^{(-2s+1/2)(-\Phi -\Phi ^2/2-\cdots )}
=-e^{\pi}e^{-(\log s)/4s+\alpha ^*/s+\cdots}$$
with $\alpha ^*=-\pi /4-2b^*+\pi ^2/4\in 
[-1.830603128\ldots ,-1.788936462\ldots ]$.


\begin{thebibliography}{40}
\bibitem{AnBe} G. E. Andrews, B. C. Berndt,  
Ramanujan's lost notebook. Part II. Springer, NY, 2009.
\bibitem{BeKi} B. C. Berndt, B. Kim, 
Asymptotic expansions of certain partial theta functions. 
Proc. Amer. Math. Soc. 139 (2011), no. 11, 3779--3788.
\bibitem{BrFoRh} K. Bringmann, A. Folsom, R. C. Rhoades, 
Partial theta functions 
and mock modular forms as $q$-hypergeometric series, Ramanujan J. 29 (2012), 
no. 1-3, 295--310,  
http://arxiv.org/abs/1109.6560
\bibitem{Ha} G.~H.~Hardy, On the zeros of a class of integral functions, 
Messenger of Mathematics, 34 (1904), 97--101. 
\bibitem{Hu} J. I. Hutchinson, On a remarkable class of entire functions, 
Trans. Amer. Math. Soc. 25 (1923), pp. 325--332.
\bibitem{KaLoVi} O.M. Katkova, T. Lobova and A.M. Vishnyakova, On power series 
having sections with only real zeros. Comput. Methods Funct. Theory 3 (2003), 
no. 2, 425--441.
\bibitem{Ka} B. Katsnelson, On summation of the Taylor series 
for the function 
$1/(1-z)$ by the theta summation method, arXiv:1402.4629v1 [math.CA]. 
\bibitem{Ko1} V.P. Kostov, About a partial theta function, 
Comptes Rendus Acad. Sci. Bulgare 66, No 5 (2013) 629-634.
\bibitem{Ko2} V.P. Kostov, On the zeros of a partial theta function, Bull. 
Sci. Math. 137, No. 8 
(2013) 1018-1030..
\bibitem{Ko3} V.P. Kostov, Asymptotics of the spectrum of partial theta 
function, Revista Mat. Complut. 27, No. 2 (2014) 677-684, 
DOI: 10.1007/s13163-013-0133-3.
\bibitem{Ko4} V.P. Kostov, A property of a partial theta function, 
Comptes Rendus Acad. Sci. Bulgare (to appear).
\bibitem{KoSh} V.P. Kostov and B. Shapiro, Hardy-Petrovitch-Hutchinson's 
problem and partial theta function, Duke Math. J. 162, No. 5 (2013) 825-861, 
arXiv:1106.6262v1[math.CA].
%\bibitem{Le} B. Ja. Levin, Zeros of Entire Functions, AMS, Providence, 
%Rhode Island 1964.
\bibitem{Ost} I.~V.~Ostrovskii, On zero distribution of sections and tails 
of power series, Israel Math. Conf. Proceedings, 15 (2001), 297--310.
\bibitem{Pe} M.~Petrovitch, Une classe remarquable de s\'eries enti\`eres, 
Atti del IV Congresso Internationale dei Matematici, Rome (Ser. 1), 2 
(1908), 36--43. 
\bibitem{PoSz} G. P\'{o}lya, G. Szeg\H{o}, 
Problems and Theorems in Analysis, Vol. 1, Springer, Heidelberg 1976.
\bibitem{So} A.~Sokal, The leading root of the partial theta function, 
Adv. Math. 229 (2012), no. 5, 2603--2621, 
arXiv:1106.1003.
\bibitem{Wa} S. O. Warnaar, Partial theta functions. I. Beyond 
the lost notebook, Proc. London Math. Soc. (3) 87 (2003), no. 2, 363--395.
\end{thebibliography}
\end{document}